\documentclass[12pt,a4paper]{amsart}

 \usepackage{amssymb, amstext, amscd, amsmath, amsfonts, amsthm, amscd, color}

\usepackage{tikz,mathdots,enumerate}

\usepackage{graphicx, array, blindtext}
\usepackage{url}
\usepackage{tikz}
\usetikzlibrary{shapes.geometric, calc}
\usepackage{caption}
\usepackage{array}
\usepackage{epsfig}
\usepackage{eucal}
\usepackage{latexsym}
\usepackage{mathrsfs}
\usepackage{textcomp}
\usepackage{verbatim}

\newtheorem{thm}{Theorem}[section]

\newtheorem{lemma}[thm]{Lemma}

\newtheorem{prop}[thm]{Proposition}

\numberwithin{equation}{section}

\theoremstyle{definition}

\newtheorem{rem}[thm]{Remark}

\newcommand{\bC}{{\mathbb{C}}}

\newcommand{\bN}{{\mathbb{N}}}

\newcommand{\bR}{{\mathbb{R}}}
\newcommand{\bT}{{\mathbb{T}}}

  \newcommand{\A}{{\mathcal{A}}}
  \newcommand{\B}{{\mathcal{B}}}
  \newcommand{\C}{{\mathcal{C}}}

\renewcommand{\H}{{\mathcal{H}}}

  \newcommand{\K}{{\mathcal{K}}}
\renewcommand{\L}{{\mathcal{L}}}
  
  \newcommand{\N}{{\mathcal{N}}}

\renewcommand{\S}{{\mathcal{S}}}
  
  \newcommand{\U}{{\mathcal{U}}}

\newcommand{\fS}{{\mathfrak{S}}}

\usepackage{verbatim}

\usepackage{tikz}
\usetikzlibrary{calc}

\begin{document}


\title[The parabolic algebra revisited]{The parabolic algebra revisited}



\author[E. Kastis  and S. C. Power]{E. Kastis and S. C. Power}
\address{Dept.\ Math.\ Stats.\\ Lancaster University\\
Lancaster LA1 4YF \\U.K. }

\email{l.kastis@lancaster.ac.uk}
\email{s.power@lancaster.ac.uk}


\thanks{
This work was supported by the Engineering and Physical Sciences Research Council [EP/P01108X/1]}

\begin{abstract}
The parabolic algebra $\A_p$ is the weakly closed algebra on $L^2(\bR)$ generated by the unitary semigroup of right translations and
the unitary semigroup of multiplication by the analytic
exponential functions $e^{i\lambda x}, \lambda \geq 0$. This algebra is reflexive with an invariant subspace lattice, $Lat \A_p$, which is naturally homeomorphic to the unit disc (Katavolos and Power, 1997). This identification is used here to
classify  strongly irreducible isometric representations of the partial Weyl commutation relations. The notion of a synthetic subspace lattice is extended from commutative to noncommutative lattices and it is shown that $Lat \A_p$ is nonsynthetic relative to the maximal abelian multiplication subalgebra of $\A_p$.
Also, operator algebras derived from isometric representations of $\A_p$ and from compact perturbations are defined and determined.
\end{abstract}
\date{}

\maketitle

\section{Introduction}
Let $ M_\lambda$ and $D_\mu$ be the unitary operators on the Hilbert space
$L^2(\bR)$ given by
\[
M_\lambda f(x) = e^{i\lambda x}f(x),\quad D_\mu f(x) = f(x-\mu)  
\]
where  $\mu, \lambda$ are real.
 As is well-known, the $1$-parameter unitary groups
$\{D_\mu, \mu \in \bR\}$ and $\{M_\lambda, \lambda \in \bR\}$  provide an irreducible representation of the Weyl commutation relations (WCR), $M_\lambda D_\mu
= e^{i\lambda \mu} D_\mu M_\lambda$, and the weak operator topology  closed operator algebra that they generate is the von Neumann algebra  $B(L^2(\bR))$ of all bounded operators. 
See Taylor \cite{tay}, for example. 
On the other hand Katavolos and Power \cite{kat-pow-1} considered the weakly closed operator algebra  generated by the unitary semigroups, for $\mu \geq 0$ and $ \lambda \geq 0$, and showed it to be a proper subalgebra, containing no self-adjoint operators, other than real multiples of the identity, and no nonzero finite rank operators. Moreover this operator algebra, the \emph{parabolic algebra} $\A_p$, was shown to be reflexive, in the sense $\A_p= Alg Lat\A_p$, with the set of invariant subspaces naturally homeomorphic to a closed disc.
However, despite this progress basic algebraic questions remain unanswered, such as the structure of closed ideals and whether zero divisors exist. 

In what follows we revisit the parabolic algebra and its noncommutative  invariant subspace lattice and we examine  associated operator algebras arising from semigroups of isometries 
and from compact perturbations. Also, an isometries generalisation of the Stone-von Neumann uniqueness theorem is obtained by making use of the identification of $Lat \A_p$.

Recall that the Stone-von Neumann theorem provides a complete classification of the pairs of strongly continuous unitary groups acting on a separable Hilbert space which satisfy the Weyl commutation relations \cite{sto}, \cite{neu}. Specifically, there is one irreducible class, modelled by translation and multiplication operators on $L^2(\bR)$, and the finite and countable direct sums of this representation determine the other unitary equivalence classes. Rosenberg \cite{ros} has given an interesting historical perspective on the origins of this result, whose strict proof was completed by von Neumann  \cite{neu} in 1931.
It is possibly well-known that a strongly continuous isometric representation of the (partial) Weyl commutations relations for two semigroups of isometries may be dilated to a unique minimal strongly continuous unitary representation of the (full) Weyl commutation relations. A simple proof is given in Theorem \ref{t:dilatingisometric}. However we are not aware of a unitary equivalence class classification for such pairs of semigroups and we obtain partial results here.
We show that the classes which are \emph{strongly irreducible}, in the sense that their unique minimal unitary dilations are irreducible, are  parametrised by the closed unit disc with a boundary point removed.
 
Recall also, that an invariant subspace lattice $\L$ of an operator algebra is a commutative subspace lattice (CSL) if its associated projections form a commuting family. Arveson \cite{arv} has defined such a lattice to be \emph{synthetic} if $Alg \L$ coincides with a certain minimal weak$^*$-closed algebra $\A_{\rm min}$ constructed directly from pseudo-integral operators associated with $\L$. Less technical is the equivalent property that $Alg \L$  is the unique weak$^*$-closed algebra $\A$ such that  $Lat \A=\L$ and $\A$ contains a maximal abelian self-adjoint algebra associated with $\L$. It was shown moreover that this notion of synthesis is
related to sets of spectral synthesis in harmonic analysis, and that CSLs failing to be synthetic could be constructed in terms of sets failing  spectral synthesis. See also Davidson \cite{dav} and Shulman and Turowska \cite{shu-tur}. On the other hand 
the continuous projection nest $\N_v$, for $L^2(\bR)$, and indeed 
any complete projection nest is synthetic.
In Section \ref{s:synthetic} we introduce an analogous notion of synthesis for a noncommutative reflexive subspace lattice $\L$, namely synthesis  relative to a maximal abelian subalgebra of $Alg\L$, and we show that $Lat \A_p$ is not synthetic relative to the maximal abelian subalgebra  $M_{H^\infty(\bR)}$ of $\A_p$.

 In the final two sections
we examine, respectively, the weakly closed operator algebras determined by the restrictions of $\A_p$ to an invariant subspace,
and the \emph{quasicompact algebra} of $\A_p$, which is the
algebra
\[
Q\A_p = (\A_p +\K)\cap (\A_p^*+ \K).
\]
This is shown to be a C*-algebra which strictly contains
$(\A_p\cap \A_p^*) +\K = \bC I +\K$.

Understanding  the algebraic and geometric  structure of the parabolic algebra presents some interesting challenges and in Section \ref{ss:openprobs} we indicate 4 natural open problems.

\section{The parabolic algebra}\label{s:parabolicalgebra}
We start by recalling basic facts and notation concerning the parabolic algebra, its subspace of Hilbert-Schmidt operators and its invariant subspaces.
 
The Volterra nest  $\mathcal{N}_v$ is the nest of subspaces $L^2([\lambda,+\infty))$, for $\lambda\in\bR$, together with the trivial subspaces $\{0\},L^2(\bR)$. The analytic nest $\N_a$ is the unitarily equivalent nest $F^\ast \N_v$, where $F$ is Fourier-Plancherel transform with
\[
Ff(x)=\frac{1}{\sqrt{2\pi}}\int_\bR f(t)e^{-itx}dt.
\] 
By the Paley-Wiener theorem the analytic nest consists of the chain of subspaces 
\[
e^{isx}H^2(\bR), \quad s \in \bR,
\] 
together with the trivial subspaces. These nests determine  the Volterra nest algebra $\A_v=Alg\N_v$ and the analytic nest algebra $\A_a=Alg\N_a$, both of which are reflexive operator algebras since they have the form $Alg(\S)$, with $\S$ a set of subspaces.

Define also the reflexive \emph{Fourier binest algebra}
$\A_{FB}= Alg(\mathcal{N}_a\cup\mathcal{N}_v) = \A_a \cap \A_v$.
The subspace lattice $\mathcal{N}_a\cup\mathcal{N}_v$, which is a continuous complete lattice with noncommuting subspace projections, is known as the \emph{Fourier binest},

 The antisymmetry property $\A_{FB}\cap\A_{FB}^\ast=\bC I$ follows readily, since $\A_v\cap A_v^*$ is the algebra of multiplication operators $M_\phi$ with $\phi \in L^\infty(\bR)$ real-valued, and these operators must leave $H^2(\bR)$ invariant. Also $\A_{FB}$ contains no non-zero finite rank operators. This follows from the structure of such operators in a nest algebra (see Davidson \cite{dav}) and the fact that a pair of proper subspaces from $\N_a$ and $\N_v$ have trivial intersection.

Consider now the \emph{parabolic algebra} $\A_p$. This is the weak operator topology  closed operator algebra generated by the unitary semigroups of operators $\{M_\lambda, \lambda\geq 0\}$ and $\{D_\mu,\,\mu\geq 0\}$. 
Since the generators of $\A_p$ leave the binest invariant, we have 
 $\A_p\subseteq\A_{FB}$. That these two algebras are equal was shown in Katavolos and Power \cite{kat-pow-1}. We show this here by repeating the argument of Levene \cite{lev}. 
 
Write $\C_2$ for the space of Hilbert-Schmidt operators on $\L^2(\bR)$
and recall that every such operator has the form $Int k $ for some square-summable function $k(x,y)$ in $L^2(\bR^2)$ where $Intk$ denotes the Hilbert-Schmidt operator acting on $L^2(\bR)$ given by 
\begin{align*}
(Intk\,f)(x)=\int_\bR k(x,y)f(y)dy.
\end{align*}
The following identification of the Hilbert-Schmidt operators in the Fourier binest algebra is a straightforward argument using the Fourier-Plancherel  transform. The space $H^2(\bR)\otimes L^2(\bR_+)$ is the Hilbert space tensor product.

\begin{prop}\label{p:integralop}
For $k\in L^2(\bR^2)$ let $\Theta_p(k)(x,t) = k(x,x-t)$. Then
\begin{align*}
\A_{FB}\cap \mathcal{C}_2\subseteq\{Intk\,|\,\Theta_p(k)\in H^2(\bR)\otimes L^2(\bR_+)\}
\end{align*}
\end{prop}
Now, given $h\in H^\infty(\bR)\cap H^2(\bR), \phi\in L^1\cap L^2(\bR_+)$, let $h\otimes \phi$ denote the function $(x,y)\mapsto h(x)\phi(y)$. The integral operator $Int k$, that is induced by the function $k=\Theta_p^{-1}(h\otimes\phi)$, lies in the parabolic algebra. Indeed, we have
$Int k=M_h\Delta_\phi$, where $\Delta_\phi$ is the bounded operator  defined by the sesquilinear form
\begin{align*}
\langle\Delta_\phi f,g\rangle=\int_\bR\int_\bR \phi(t)D_tf(x)\overline{g(x)}dxdt, \textrm{ where }f,g\in L^2(\bR).
\end{align*}
Since the linear span of these separated variable
 functions $k$ is dense in $H^2(\bR)\otimes L^2(\bR_+)$, it follows from  Proposition \ref{p:integralop} that 
\begin{align*}
\A_{FB}\cap \mathcal{C}_2\subseteq\{Intk\,|\,\Theta_p(k)\in H^2(\bR)\otimes L^2(\bR_+)\}\subseteq \A_p\cap\mathcal{C}_2\subseteq \A_{FB}\cap \mathcal{C}_2
\end{align*}
and so $\A_{FB}\cap \mathcal{C}_2 =  \A_{FB}\cap \mathcal{C}_2$.

\begin{thm} The parabolic algebra coincides with the Fourier binest algebra and  is a reflexive operator algebra.
\end{thm}

\begin{proof}  By considering an appropriate sequence of operators of the form $M_h\Delta_\phi$ it can be shown that
$\A_p$ has a bounded approximate identity of Hilbert-Schmidt operators for the strong operator topology. Since
$\A_p$ and $\A_{FB}$ have the same Hilbert-Schmidt operators a simple density argument completes the proof.
\end{proof}

Let us now consider the invariant subspace lattice 
$Lat \A_p$.
In \cite{kat-pow-1} a cocycle argument with inner functions and unimodular functions on the line is used to obtain the following identification:
\begin{align*}
Lat\A_p=\{K_{\lambda,s}|\lambda\in\bR,s\geq 0\}\cup \mathcal{N}_v
\end{align*}
where $K_{\lambda,s}=M_\lambda M_{\phi_s}H^2(\bR)$ and $\phi_s(x)=e^{-isx^2/2}$. (A similar argument is used in the Section \ref{s:synthetic}.)
Thus for $s>0$ we have the nest 
$\mathcal{N}_s=M_{\phi_s}\mathcal{N}_a$, and for  distinct values of $s\geq 0$ these are disjoint, except for the trivial subspaces.
If we view $Lat\A_p$ as a set of projections endowed with the strong operator topology, then it is homeomorphic to the closed unit disc and the topological boundary is the Fourier binest. 


\begin{figure}[h!]
\begin{center}
\begin{tikzpicture}
\filldraw [black] (3,6) circle (1pt);
\draw (3,6) node [above]  {$L^2(\bR)$};
\filldraw [black] (3,0) circle (1pt);
\draw (3,0) node [below] {$\{0\}$};
\filldraw [black] (0,3) circle (1pt);
\draw (0,3) node [left]  {$H^2(\bR)$};
\filldraw [black] (6,3) circle (1pt);
\draw (6,3) node [right]  {$L^2(\bR_+)$};
\filldraw [black] (1.65,3) circle (1pt);
\draw (1.65,3) node [right] {$M_{\phi_s}H^2(\bR)$};
\filldraw [black] (0.35,1.59) circle (1pt);
\draw (4,3) node [right] {$s\rightarrow +\infty$};
\draw (0.35,1.59) node [left]  {$M_{\lambda}H^2(\bR)$};
\filldraw [black] (5.65,1.59) circle (1pt);
\draw (5.65,1.59) node [right]  {$D_\mu L^2(\bR_+)$};
\filldraw [black] (1.9,1.59) circle (1pt);
\draw (1.9,1.59) node [right]  {$M_{\phi_s}M_\lambda H^2(\bR)$};
\draw (0.35,4.41) node [above left] {$\N_a$};
\draw (1.9,4.41) node [above left] {$\N_s$};
\draw (5.65,4.41) node [above right] {$\N_v$};
\draw (3,3) circle [radius=3.0];
\draw (3,0) to [out=140,in=220] (3,6);
\end{tikzpicture}
\end{center}
\caption{Parametrising $Lat \A_p$ by the unit disc.}
\label{Latpar}
\end{figure}
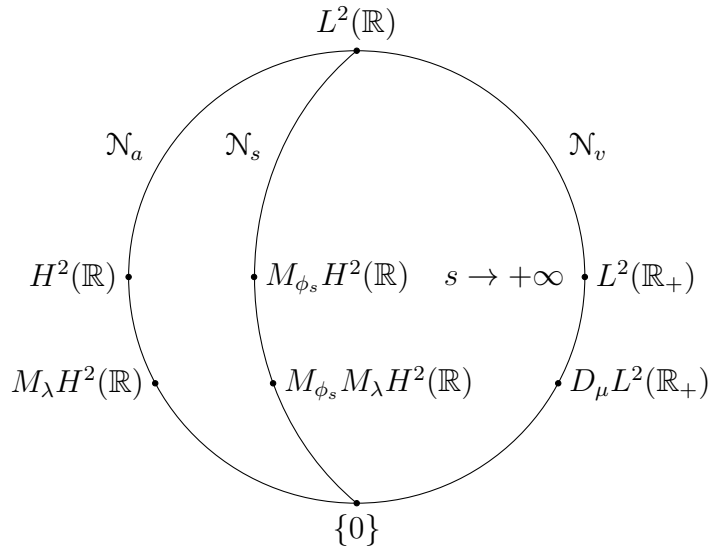

It follows in particular that the Fourier binest lattice is not  reflexive. That is, the lattice $Lat Alg(\N_a\cup\N_b)$ strictly contains $\N_a\cup \N_b$.

We note that similar results have been obtained by Kastis \cite{kas-Lp} for the strong operator topology closed operator algebras on $L^q(\bR)$ generated by the corresponding shift isometries and multiplication isometries of   $L^q(\bR)$, for $1 <q<\infty, q\neq 2$. However, it is not known whether there is a similar homeomorphism between their invariant subspace lattices and the closed unit disc.

\subsection{The operator algebras of $(\lambda, \mu)$-cones} We next consider the operator algebra analogues of $\A_p$ associated with cones in $\bR^2$ in place of the first quadrant cone of the parameters  $\lambda, \mu$.
 
Let $\mathfrak{S}$ be a cone in $\bR^2$, that is, an additive semigroup containing 0 with the additional property that $r(\lambda,\mu)\in\mathfrak{S}$, for all $r>0$ and $(\lambda,\mu)\in\mathfrak{S}$. 
Define $\A_\mathfrak{S}$ to be the $w^*$-closed linear span of the unitaries $U_\lambda V_\mu$, for $(\lambda,\mu)\in\mathfrak{S}$. This is an operator algebra and is equal to $\A_p$ if $\fS= \bR_+^2$.
We now show that for any simple cone $\fS$, that is, one determined by an acute or obtuse angle, we have a unitary equivalence between $\A_\fS$ and  $\A_p$.
The following parametrisation will be convenient. 
With $s_1,s_2\in\bR$ let $\fS_{s_1,s_2}$ be the additive cone for distinct rays, $(x, s_1x), x \geq 0,$ and  $(x, s_2x),x \geq 0$, 
 and let $\A_{s_1,s_2}=\A_{\mathfrak{S}_{s_1,s_2}}$. Also we write  $Ad_Z$ for the map $X \to ZXZ^*$ determined by a unitary operator $Z$ and an associated domain of operators.

\begin{thm} \label{t:conealgebras}
Let $\mathfrak{S}$ be a simple cone in $\bR^2$. Then  $\A_\fS$ is unitarily equivalent to the parabolic algebra.
\end{thm}
\begin{proof}
Since $Ad_F D_\mu={M}_{-\mu}$, without loss of generality we consider cones $\mathfrak{S}_{0,s}$ and $\mathfrak{S}_{s_1,s_2}$, with $s,s_1>0$.
In particular, it suffices to prove that the corresponding algebras $\A_{s}$ and $\A_{s_1,s_2}$ (see Figure~\ref{fig2}) are unitarily equivalent to the parabolic algebra.
\begin{figure}[h]
\begin{center}
\begin{tikzpicture}
\path [fill=cyan] (1.75,1) to (3.5,1) -- (3.5,1) to (3.5,2.5) -- (3.5,2.5) to (1.75,2.5) --(1.75,2.5) to (1.75,1);
\draw [->][thick] (1.75,1) -- (3.5,1);
\draw [<-] (0,1) -- (1.75,1);
\draw (3.3,1) node [below] {$M_\lambda$};
\draw [<-] (1.75,0) -- (1.75,1);
\draw [->][thick] (1.75,1) -- (1.75,2.5);
\draw (1.75, 2.4) node [left] {$D_\mu$};
\draw (2.5,2) node [right] {$\A_p$};

\path [fill=cyan] (6.75,1) to (8.5,1) -- (8.5,1) to (8.5,2.1) -- (8.5,2.1) to (6.75,1);
\draw [<-] (5,1)--(6.75,1);
\draw [->] [thick] (6.75,1) -- (8.5,1);
\draw [->] [thick] (6.75,1) -- (8.5,2.1);
\draw (8.3,1) node [below] {$M_\lambda$};
\draw [<->] (6.75,0) -- (6.75,2.5);
\draw (6.75, 2.4) node [left] {$D_\mu$};
\draw (7.75,1) node [above right] {$\A_s$};
\draw [dashed] (7.25,1) -- (7.25,1.3);
\draw (7.25,1) node [below] {$_{M_1}$};
\draw [dashed] (7.25,1.3) -- (6.75,1.3);
\draw (6.75,1.3) node [left] {$_{D_s}$};

\path [fill=cyan] (11.75,1) to (12.5,2.5) -- (12.5,2.5)  to (10,2.5) -- (10,2.5) to (10,1.5) -- (10,1.5) to (11.75,1);
\draw [->] [thick] (11.75,1) -- (12.5,2.5);
\draw [->] [thick] (11.75,1) -- (10,1.5);
\draw [<->] (10,1) -- (13.5,1);
\draw (13.3,1) node [below] {$M_\lambda$};
\draw [<->] (11.75,0) -- (11.75,2.5);
\draw (11.75, 2.4) node [left] {$D_\mu$};
\draw (11,2) node [left] {$\A_{s_1,s_2}$};

\draw [dashed] (12.25,1) -- (12.25,2);
\draw (12.25,1) node [below] {$_{M_1}$};
\draw [dashed] (11.75,2) -- (12.25,2);
\draw (11.75,2) node [left] {$_{D_{s_1}}$};

\draw [dashed] (11.25,1) -- (11.25,1.15);
\draw (11.25,1) node [below] {$_{M_{-1}}$};
\draw [dashed] (11.75,1.15) -- (11.25,1.15);
\draw (11.75,1.15) node [above left] {$_{D_{s_2}}$};

\end{tikzpicture}
\end{center}
\caption{From the left to the right : The parabolic algebra $\A_p$, the algebra $\A_{s}$ and the algebra $\A_{s_1,s_2}$}
\label{fig2}
\end{figure}
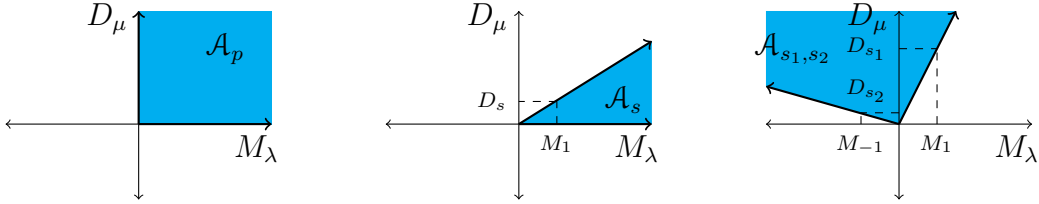

Let $s\in\bR$. Given $\mu\in\bR$ and $f\in L^2(\bR)$, compute
\begin{align*}
(D_\mu M_{\phi_s}f)(x)&=(M_{\phi_s}f)(x-\mu)=
e^{-is(x-\mu)^2/2}f(x-\mu)\\&=
e^{-is\mu^2/2}e^{-isx^2/2}e^{is\mu x }f(x-\mu)=
e^{-is\mu^2/2}(M_{\phi_s}M_{s\mu}D_\mu f)(x),
\end{align*}
which implies 
\begin{equation}\label{eq1}
Ad_{M_{\phi_s}}(M_{s\mu}D_\mu)=e^{-is\mu^2/2} D_\mu.
\end{equation}

Also, applying $Ad_F$, we get
\begin{align*}
Ad_{D_{\phi_s}}(D_{s\mu}M_{-\mu})=e^{-is\mu^2/2} M_{-\mu}
\end{align*}
or equivalently
\begin{equation}\label{eq2}
Ad_{D_{\phi_{s^{-1}}}^\ast}(D_{\mu}M_{s\mu})=e^{-is\mu^2/2} M_{s\mu}.
\end{equation}
Thus, $\A_s$ is unitarily equivalent to the parabolic algebra, since, by equation~\eqref{eq1}, the map $Ad_{M_{\phi_s}}$ sends $\A_s$ onto $\A_p$. In the general case, let $s_1>s_2$. Then, applying the formula~\eqref{eq2}, the restriction of the map $Ad_Z$, with $Z={D_{\phi_{s_2^{-1}}}^\ast}$, to the domain $\A_{s_1,s_2}$, is an isomorphism onto $\A_{s_1-s_2}$, and so the theorem follows.  
\end{proof}

\section{Synthetic lattices and Lat$\A_p$.}\label{s:synthetic} 
Let $\L$ be a reflexive subspace lattice in the usual sense of Halmos, so that $\L= Lat Alg \L$, and let $\B$ be a maximal abelian subalgebra of $Alg \L$. Then we define $\L$ to be \emph{synthetic relative to  $\B$}  if $Alg \L$ is the unique weak$^*$-closed operator algebra $\A$ with  $Lat \A = \L$ and $\B\subseteq \A$. 

As we have remarked in the introduction, in the case of a commutative subspace lattice $\L$ being synthetic relative to a maximal abelian self-adjoint subalgebra of $Alg \L$
is equivalent to Arveson's more technical definition of a synthetic CSL lattice. In this formulation $Alg \L$ must coincide with a certain \emph{minimal algebra} determined by so-called pseudo-integral operators constructed with the help of a coordinatisation of $\L$. These generating operators play a role analogous to the Hilbert-Schmidt operators in a nest algebra. 

The parabolic algebra is known to be the weak$^*$-closure of its Hilbert-Schmidt integral operators \cite{kat-pow-1}. Despite this regularity property we now show that its invariant subspace lattice  fails to be synthetic relative to the maximal abelian algebra of analytic multiplication operators.

\begin{thm}\label{t:notsynthetic}
The subspace lattice $Lat \A_p$ is not synthetic relative to $M_{H^\infty(\bR)}.$
\end{thm}

\begin{proof}
Let $\A_0$ be the subalgebra of $\A_p$  generated by
the operators $M_\lambda$, for $\lambda\geq 0$, and the products $M_\lambda D_\mu$ where $\lambda \geq 2$ (and $\mu\geq 0$) or $\mu\geq 2$ (and $\lambda\geq 0$), and let $\A$ be the weak operator topology closure of $\A_0$.
In the first part of the proof we show that $Lat \A = Lat \A_p$ by means of the  inner function cocyle argument similar to that used for the determination of $Lat \A_p$ in \cite{kat-pow-1}. In the second part of the proof we show that $\A \neq \A_p$.

Let $K \in Lat \A$.
Since $M_\lambda K\subseteq K$ for all $\lambda\geq 0$, it follows from Beurling's theorem that $K=uH^2(\bR)$, for some unimodular function $u$, or $K=L^2(E)$, where $E$ is a Borel set in $\bR$.
If $K=L^2(E)$ then, since it is an invariant subspace for  $D_\mu M_\lambda$, for $\mu\geq 0, \lambda \geq 2$, it follows that $E=[t,\infty)$, for some $t\in \bR$. 

Suppose now that $K=uH^2(\bR)$. Fix $\lambda >2$. Then $M_\lambda D_\mu$ is in $\A_0$, for all $\lambda\geq 0$, and  $ M_\lambda D_\mu K\subseteq K$. {Also $M_\lambda D_\mu  K$ is invariant under $M_{\lambda'}$, for all $\lambda'\geq 0$}, and so, by Beurling's theorem, $D_\mu M_\lambda  K= \omega_{\mu}u H^2(\bR)$, for some unimodular function $\omega_{\mu}$. 
Note that $\omega_{\mu_2}$ divides $\omega_{\mu_1}$ for all $0<\mu_2<\mu_1$. 
Moreover, calculating directly we have $D_\mu M_\lambda  K=u(x-\mu)e^{i\lambda x}H^2(\bR)$. Hence we obtain $u(x-\mu)e^{i\lambda x}= c_{\mu} \omega_{\mu} u(x)$, for some unimodular constant. Redefining  $\omega_{\mu}$ we may assume that $c_{\mu}=1$. Thus
\[\omega_{\mu}=\frac{ u(x-\mu)}{u(x)} e^{i\lambda x}.\]
Therefore, we get the cocycle equation
\[\omega_{\mu_1 +\mu_2}=\frac{u(x-\mu_1-\mu_2)}{u(x-\mu_1)} \frac{u(x-\mu_1)}{u(x)} e^{i\lambda x}=\omega_{\mu_2}(x-\mu_1) \omega_{\mu_2}(x) e^{-i\lambda (x-\mu_1)}.\]
Thus $\omega_{\mu_2}(x-r)$ divides $\omega_{\mu_1}$ for all $0<r<\mu_1-\mu_2$. Fix $\mu_1$ and $\mu_2$ with $0<\mu_2<\mu_1$. If $\omega_{\mu_2}$ has any zeros in the upper half plane, then those zeros and all their translates by $r$, with $0<r<\mu_1-\mu_2$, must be zeros of $\omega_{\mu_1}$. However, this would imply that the analytic function $\omega_{\mu_1}$ is identically zero, a contradiction, so $\omega_{\mu_2}$ admits a trivial Blaschke product.
Hence $\omega_{\mu_2}$ can be writen in the form
\[\omega_{\mu_2}(z)= \alpha\, e^{i\beta z} \exp\left\{i\int_\bR \frac{sz+1}{s-z}\frac{1}{s^2+1} d\mu(s)\right\}, \,\operatorname{Im}z>0,\]
for some unimodular $\alpha$, some real $\beta$ and some singular measure $\mu$.
Let $\omega_{\mu_1}$ be associated with the triple $\alpha',\beta'$ and $\nu$. Again,  $\omega_{\mu_2}(x-r)$  divides $\omega_{\mu_1}$, which yields that all the translates of $\mu$ by $r$, with $0<r<\mu_1-\mu_2$, are dominated by the singular measure $\nu$, and so $\mu=0$. Thus 
\[\omega_{\mu_2}(x)=\alpha_{\mu_2}e^{i\beta(\mu_2)x},\]
where $\alpha$ is unimodular and $\beta$ is strictly increasing. By the cocycle equation, we have 
\[\alpha(\mu_1+\mu_2) e^{i\beta(\mu_1+\mu_2)x}=\alpha(\mu_2)e^{i\beta(\mu_2)(x-\mu_1)}
\alpha(\mu_2)e^{i\beta(\mu_2)x} e^{-i\lambda(x-\mu_1)},\]
which implies that  $\alpha(\mu_1+\mu_2)=\alpha(\mu_1)\alpha(\mu_2)e^{-i\beta(\mu_2)\mu_1}e^{i\lambda \mu_1}$ and $\beta(\mu_1+\mu_2)=\beta(\mu_1)+\beta(\mu_2)-\lambda$. Thus $\beta(\mu_2)=\rho \mu_2+\lambda$, for some $\rho>0$. Therefore, $\alpha(\mu_2)=e^{i\sigma \mu_2}e^{-i\rho \mu_2^2/2}$. Hence
\[\omega_{\mu_2}=e^{i\sigma \mu_2}e^{-i\rho \mu_2^2/2}e^{i(\rho\mu_2+\lambda)x}=\frac{u(x-\mu_2)}{u(x)}e^{i\lambda x}.\]
Therefore, fixing $x=x_0$ we have
\[u(x_0-\mu_2)=u(x_0) e^{i(-\rho \mu_2^2/2+\sigma\mu_2+\rho\mu_2 x_0)}.\]
Thus
\[u(y)=c e^{i(-\rho y^2/2+\sigma\mu_2)}\]
and the first part of the proof is complete.


Note that $\A$ is also the closure of the subalgebra 
$M_{H^\infty(\bR)}+M_2\A_p +D_2\A_p$ and this is contained in
$\A_v$, the Volterra  nest algebra of operators with ``lower triangular support". The following rough argument based on support sets shows that  $\A \neq \A_p$ and provides insight for the more explicit separation argument below, showing that $D_1 \notin \A_p$.
The support of $D_\mu$, with $0<\mu<2$, is a line parallel to the main diagonal. This main diagonal is the support set for $M_{H^\infty(\bR)}$ and the support set for $D_2\A_p$ is empty "above" the support for $D_2$. It follows from this that if $D_\mu$ is in $\A$ then it must be  in the weak operator topology closure of 
$M_2\A_p$. Since $M_2\A_p$ is closed in this topology it  follows that $D_\mu\in M_2\A_p$. However, this cannot be true for all such $\mu$ since this implies that $I\in M_2\A_p$, a contradiction.

Let $f_n,g_n\in L^2[0,2]$, with $\|f_n\|_2 = \|g_n\|_2 =1$, let 
$(\alpha_n)$ be a summable sequence, and define the associated  w$^*$-continuous functional $\omega$ on $B(L^2(\bR))$ with
\[\omega(T)=\sum_{n\in\bN}\alpha_n\langle TM_{\phi_1}F^{-1}f_n,M_{\phi_1}F^{-1}g_n\rangle.
\]

Consider now a finite sum $\sum_m c_m M_{\lambda_m}D_{\mu_m}$ in $\A_0$, with $\lambda_m+\mu_m\geq 2$ for every $m$, and $h\in H^\infty(\bR)$.
We have
\begin{align*}
\omega(M_{h}+ \sum_m c_m M_{\lambda_m}D_{\mu_m}+D_1)&=
\sum_{n\in\bN}\alpha_n\langle (M_{h}+ \sum_m c_mM_{\lambda_m}D_{\mu_m}+D_1)
M_{\phi_1}F^{-1}f_n,M_{\phi_1}F^{-1}g_n\rangle
\end{align*}
\begin{align*}
{\color{white}xxxxxx}&=\sum_{n\in\bN}\alpha_n\langle M_{\phi_1}F^{-1}(D_{h}+ \sum_m c_me^{-m^2/2}D_{\lambda_m}D_{\mu_m}M_{-\mu_m}+e^{i/2}D_1M_{-1})
f_n, M_{\phi_1}F^{-1}g_n\rangle\\
&=\sum_{n\in\bN}\alpha_n\langle M_{\phi_1}F^{-1}(D_{h}+ \sum_m c_me^{-m^2/2}D_{\lambda_m + \mu_m}M_{-\mu_m}+e^{i/2}D_1M_{-1})
f_n, M_{\phi_1}F^{-1}g_n\rangle\\
&= \sum_{n\in\bN}\alpha_n\langle (D_h+e^{-i/2}D_1M_{-1})f_n,g_n\rangle,
\end{align*}
{The last equality follows from the fact that for all $f,g\in L^2[0,2]$ we have $\langle D_\lambda f, g\rangle=0$, when $\lambda\geq 2$.
}

We may write the previous equality as
\[
\omega(M_{h}+ \sum_m c_mM_{\lambda_m}D_{\mu_n}+D_1)=\omega_{[0,2]}(D_h+e^{-i/2}D_1M_{-1})
\]
where $\omega_{[0,2]}$ denotes the $w^*$-continuous functional on $B(L^2[0,2])$ determined by the sequences $(f_n) $ and $(g_n)$, and we note that every  $w^*$-continuous functional on $B(L^2[0,2])$ is of this form.

Now the operator $D_1M_{-1}$ does not lie in the $w^*$-closure of the compression of $\{D_h:h\in H^\infty(\bR)\}$ to $L^2[0,2]$. Indeed, the latter algebra is commutative while 
the compression of $D_1M_{-1}$ does not commute with the compression of  $D_{1/2}$.
Thus, by the topological Hahn-Banach theorem, there exists a $w^*$-continuous 
functional $\omega_{[0,2]}$ on $B(L^2[0,2])$, such that $|\omega_{[0,2]} (e^{-i/2}D_1M_{-1})|=1$, while $\omega_{
[0,2]}(D_h)=0$ for every $h\in H^\infty(\bR)$. However, from the calculations  above, this implies that there  
exists a $w^*$-continuous functional $\omega$ on $B(L^2(\bR))$, such that $\omega(M_{h}+ \sum_m a_mM_{\lambda_m}D
_{\mu_n})=0$ and $|\omega(D_1)|=1$. Thus  
$D_1\notin \A$ and the proof is complete.
\end{proof}

\begin{rem}
While synthesis relative to a maximal abelian algebra generalises the CSL notion of synthetic we note that for a noncommutative lattice $\L$ there may be no counterpart to the notion of a minimal weak$^*$-closed subalgebra. To see this consider the subalgebras
$\A_t= M_{H^\infty(\bR)}+M_t\A_p +D_t\A_p$, for $t>0$, which form a decreasing chain of weak$^*$-closed subalgebras with intersection equal to $M_{H^\infty(\bR)}$. The proof above shows that  $Lat \A_t= Lat \A_p$, for all $t$, and yet $Lat M_{H^\infty(\bR)}\neq Lat \A_p$.

We remark that the stronger synthesis property for a reflexive lattice $\L$, which requires that $\A= Alg\L$ for \emph{every} weak$^*$-closed unital subalgebra $\A$  with  $Lat \A = Alg \L$ is a distinctly stronger notion. Indeed, the discrete nest and the Volterra nest fail to have this uniqueness property since these lattices can be attained by a single operator, and hence by an abelian  weak$^*$-closed unital subalgebra. See \cite{dav}, \cite{rad-ros} for such unicellular operators. 

On the other hand we remark that 
$H^\infty(\bR)$, as an operator algebra on $L^2(\bR)$, or even as the Toeplitz operator algebra on $H^2(\bR)$, does have this property by virtue of being \emph{hereditarily reflexive}. This in turn is a consequence of the fact that dual space functionals are implementable by rank one operators. See, for example, Davidson \cite{dav-distance} and Hadwin \cite{had}.
\end{rem}

\section{Isometric Representations of the WCR}\label{s:repsOfWCR}
Let us define an \emph{isometric representation} of the Weyl commutation relations to be a pair of 
strongly continuous semigroups of isometries $U_\lambda, V_\mu, \lambda,\mu \geq 0$, acting on a separable Hilbert space $\mathcal{K}$,
with
\[
U_\lambda V_\mu = e^{i\lambda\mu}V_\mu U_\lambda, \quad \lambda, \mu \geq 0.
\]
An isometric representation is \emph{irreducible} if there is no proper reducing closed subspace for the representation, and we say that it is \emph{strongly irreducible} if it has a minimal unitary dilation on a separable Hilbert space which is irreducible.

Let $\rho_{\lambda, s}$ be the isometric representation arising from
the restriction of the multiplication operators $M_\lambda$ and translation operators $D_\mu$, for $\lambda, \mu \geq 0$, to the closed subspace $K_{\lambda, s}$, for $\lambda \in \bR, s> 0$. Also, for $\lambda \in \bR$, let
$\rho_\lambda$ be given by restriction to $M_\lambda H^2(\bR)$, and for $\mu\in \bR$ let $\rho^\mu$ be given by restriction to $D_\mu L^2(\bR_+)$. Finally, let $\rho_{id}$ be the identity representation.

Recall that a semigroup of isometries is said to be \emph{pure} if the intersection of the ranges of the isometries is the zero subspace.  We say that the isometric representation  $\rho$ is of type $uu, up, pu,$ or $ pp$ if the semigroups of isometries $\{U_\lambda\}, \{V_\mu\}$ are, respectively, (i) unitary semigroups, (ii) a unitary semigroup and a pure semigroup,  (iii) a pure semigroup and unitary semigroup,  and (iv) pure semigroups. In particular $\rho$ is of  type $pp$ if the intersection of the spaces $U_\lambda \mathcal{K}$ for $\lambda \geq 0$ is the zero subspace, and
the intersection of the spaces $V_\mu \mathcal{K}$ for $\mu \geq 0$ is the zero subspace.

The Stone-von Neumann uniqueness theorem for unitary groups satisfying the Weyl commutation relations ensures that the strongly irreducible isometric representations $\rho$ of type $uu$ are unitarily equivalent to $\rho_{\rm id}$. More generally we obtain the following classification.

\begin{thm}\label{mainthrm}
Let $\rho$ be an isometric representation of the Weyl commutation relations
which is strongly irreducible and which is not of type $uu$. Then $\rho$ satisfies one of the following 3 equivalences.
\medskip

(i) $\rho$ is of type $pu$ and is unitarily equivalent to $\rho_\lambda$ for some $\lambda \in \bR$,

(ii) $\rho$ is of type $up$ and is unitarily equivalent to $\rho^\mu$ for some $\mu \in \bR$,

(iii) $\rho$ is of type $pp$ and is unitarily equivalent to $\rho_{\lambda,s}$ for some $\lambda \in \bR, s>0$.
\medskip

\noindent Moreover, any pair of representations of the same type are not unitarily equivalent.

\end{thm}

The proof has 3 ingredients, namely, the unitary dilation  of isometric representations of the Weyl commutation relations, the Stone-von Neumann theorem,  and the nature of the closed invariant subspaces for the model representation $\rho_{id}$.
The following dilation theorem, for the first step here, is perhaps well-known. The proof we give is a simple variation of the proof of the dilation of commuting isometries given in Paulsen \cite{pau}.

\begin{thm}\label{t:dilatingisometric}
Let $\{S_\lambda: \lambda\geq 0\}, \{T_\mu:\mu\geq 0\}$  be an
isometric representation of the Weyl commutation relations on the Hilbert space $\H$.
Then there is  a Hilbert space $\mathcal{K}$ and unitary representations  $\{U_\lambda : \lambda\in \bR\}, \{V_\mu:\mu\in \bR\}$, acting on $\mathcal{K}$, with $P_{\mathcal{H}} U_\lambda\big|_{\mathcal{H}}=S_\lambda$ and $  P_{\mathcal{H}} V_\mu\big|_{\mathcal{H}}=T_\mu$, for  $\lambda,\mu\geq 0$.
\end{thm}

\begin{proof}
By Naimark's theorem (\cite{pau} Theorems 4.8 and 5.4), there is a Hilbert space $\mathcal{K}_1$ containing $\mathcal{H}$ and a strong operator topology continuous (SOT-continuous) unitary group $\{\tilde{S}_\lambda :\lambda\in\bR\}$ such that $P_{\mathcal{H}}\tilde{S}_\lambda\big|_{\mathcal{H}}=S_\lambda$ for every $\lambda\geq 0$. We may assume that this is a minimal dilation of $\{S_\lambda :\lambda\geq 0\}$, that is,  the linear span of $\{\tilde{S}_\lambda h :\lambda\in \bR,\, h\in \mathcal{H}\}$ is dense in $\mathcal{K}_1$.

Fix some $\mu\geq0$ and define :
\begin{align*}
\tilde{T}_\mu\left(\sum_{n=1}^N \tilde{S}_{\lambda_n} h_n\right)=\sum_{n=1}^N e^{i\lambda_n \mu}\tilde{S}_{\lambda_n}T_\mu h_n.
\end{align*}
We claim that $\tilde{T}_\mu$ is well-defined and isometric. To see this note that given $h=\sum\limits_{n=1}^N \tilde{S}_{\lambda_n} h_n$
we have
\begin{align*}
\|\tilde{T}_{\mu} h\|^2&=
\left\langle \tilde{T}_\mu\left(\sum_{n=1}^N \tilde{S}_{\lambda_n} h_n\right),
\tilde{T}_\mu\left(\sum_{m=1}^N \tilde{S}_{\lambda_m} h_m\right)\right\rangle \\&=
\left\langle \sum_{n=1}^N e^{i\lambda_n \mu}\tilde{S}_{\lambda_n}T_\mu h_n, 
\sum_{m=1}^N e^{i\lambda_m \mu}\tilde{S}_{\lambda_m}T_\mu h_m \right\rangle  \\& =
\sum_{\lambda_n\geq \lambda_m}\left\langle e^{i\lambda_n \mu}\tilde{S}_{\lambda_n}T_\mu h_n,
e^{i\lambda_m \mu}\tilde{S}_{\lambda_m}T_\mu h_m \right\rangle +
\sum_{\lambda_m> \lambda_n}\left\langle e^{i\lambda_n \mu}\tilde{S}_{\lambda_n}T_\mu h_n,
e^{i\lambda_m \mu}\tilde{S}_{\lambda_m}T_\mu h_m \right\rangle\\&=
\sum_{\lambda_n\geq \lambda_m}\left\langle e^{i(\lambda_n - \lambda_m) \mu}\tilde{S}_{\lambda_n-\lambda_m}T_\mu h_n,
T_\mu h_m \right\rangle +
\sum_{\lambda_m> \lambda_n}\left\langle T_\mu h_n,
e^{i(\lambda_m-\lambda_n) \mu}\tilde{S}_{\lambda_m-\lambda_n)}T_\mu h_m \right\rangle \\&=
\sum_{\lambda_n\geq \lambda_m}\left\langle e^{i(\lambda_n - \lambda_m) \mu}S_{\lambda_n-\lambda m}T_\mu h_n,
T_\mu h_m \right\rangle +
\sum_{\lambda_m> \lambda_n}\left\langle T_\mu h_n,
e^{i(\lambda_m-\lambda_n) \mu}S_{\lambda_m-\lambda_n}T_\mu h_m \right\rangle \\&=
\sum_{\lambda_n\geq \lambda_m}\left\langle T_\mu S_{\lambda_n-\lambda_m} h_n,
T_\mu h_m \right\rangle +\sum_{\lambda_m> \lambda_n}\left\langle T_\mu h_n,T_\mu
S_{\lambda_m-\lambda_n} h_m \right\rangle \\&=
\sum_{\lambda_n\geq \lambda_m}\left\langle S_{\lambda_n-\lambda_m} h_n,
 h_m \right\rangle +\sum_{\lambda_m> \lambda_n}\left\langle h_n,
S_{\lambda_m-\lambda_n} h_m \right\rangle \\&=
\sum_{\lambda_n\geq \lambda_m}\left\langle S_{\lambda_n-\lambda_m} h_n,
 h_m \right\rangle +\sum_{\lambda_m> \lambda_n}\left\langle h_n,
S_{\lambda_m-\lambda_n} h_m \right\rangle \\&=
\left\langle \sum_{n=1}^N \tilde{S}_{\lambda_n} h_n,
\sum_{m=1}^N \tilde{S}_{\lambda_m} h_m\right\rangle=\|h\|^2.
\end{align*}

 Also, it follows that $\{\tilde{T}_\mu:\mu\geq 0\}$ is a SOT-continuous semigroup of isometries and  that the unitary  operators 
$\tilde{S}_\lambda$ and the isometries $\tilde{T}_\mu$ satisfy the WCR for all real  $\lambda$ and for $\mu\geq 0$. 

Now consider the minimal unitary dilation $V_\mu$ of the semigroup of isometries $\tilde{T}_\mu$ on a  Hilbert space $\K$ containing $\K_1$, so that $P_{\mathcal{K}_1}V_\mu\big|_{\mathcal{K}_1}=\tilde{T}_\mu$ for $\mu\geq 0$.
We have
$\mathcal{K}=\overline{span\{V_\mu k : \mu\in\bR, k\in\mathcal{K}_1\}}^{\|\cdot\|}$ and we may define the unitary group $\{U_\lambda:\lambda\in\bR\}$ by
\[
U_\lambda\left(\sum_{n=1}^N V_{\mu_n} h_n\right)=\sum_{n=1}^N e^{i\lambda \mu_n}T_{\mu_n} \tilde{S}_\lambda h_n.
\]
That the operators $U_\lambda$ are unitary follows from the argument above and it follows that
$\{U_\lambda\}$ and $\{V_\mu\}$ give the required unitary dilation 
\end{proof}

\begin{proof}[Proof of Theorem \ref{mainthrm}]
By the previous theorem and the Stone von Neumann theorem  every strongly irreducible isometric representation $\rho$ is unitarily equivalent to a representation $\rho_\H$ obtained from the restriction of $\{M_\lambda:\lambda \geq 0\}$ and $\{D_\mu : \mu \geq 0\}$ to an invariant subspace $\H$. In particular $\H$ is a nonzero subspace in $ Lat \A_p$ and by the description of this lattice  in Section \ref{s:parabolicalgebra} the subspace $\H$  takes one of the 4 types, 
(i) $L^2(\bR)$, (ii) $M_\lambda H^2(\bR)$, (iii) $D_\mu L^2(\bR_+)$, (iv)  $K_{\lambda,s},$ for some $\lambda \in \bR$ and  $s> 0$. 

 
Let $\rho_{\H_1}, \rho_{\H_2}$ be any two such representations which are unitarily equivalent. Then there is a unitary $Z:\H_1 \to \H_2$ such that
$Z\circ\rho_{\H_1}\circ Z^* = \rho_{\H_2}$. In particular
\[
ZM_\lambda|_{\H_1}= M_\lambda|_{\H_2}Z, \quad \mbox{ for } \lambda \geq 0.
\]
By the intertwining form of the lifting theorem for the commutant of a continuous semigroup of isometries it follows that   $Z= P_{\H_1}\tilde{Z}|_{\H_2}$ where
$\tilde{Z}$ is an operator  in the commutant of $\{M_\lambda :\lambda \geq 0\}$. (The single isometry variant of this, from which this semigroup lifting theorem may be deduced,  is due to Sarason \cite{sar-2}.) Thus $Z= M_h$ for some unimodular function $h(z)$. It now follows that
\[
M_hD_\mu|{\H_1}=D_\mu M_h|{\H_1}, \quad \mbox{ for } \mu \geq 0,
\]
from which it follows that $D_\mu h = h$ for all $\mu$. Thus $h$ is a constant function and $\H_1$ is equal to $\H_2$.
\end{proof}

\begin{rem} We note that the assumption of strong irreducibility is necessary in Theorem \ref{mainthrm}. Let $\H$ be the subspace of $L^2(\bR^2)$ given by 
\begin{align*}
\mathcal{H}=\{f\in L^2(\bR^2): f(x,y)=0, \textrm{ for a.e. } (x,y)\in \bR_-^2\}
\end{align*}
and consider the strongly continuous isometric semigroups  with
$
U_\lambda= M_\lambda\otimes D_\lambda$ and $V_\mu=D_\mu\otimes I.
$
This gives an isometric representation of the partial Weyl commutation relations which is irreducible. However, their joint  minimal unitary dilation is given by the pair of semigroups $M_\lambda\otimes D_\lambda$ and $D_\mu\otimes I$, acting on $L^2(\bR^2)$. This representation is not irreducible, since the space $L^2(\bR)\otimes H^2(\bR)$ reduces both unitary groups, and in fact is a representation with infinite multiplicity.
\end{rem}

\section{Restriction algebras}\label{s:fromIsometricReps}

Let us refer to an operator algebra of the form $\A|_K$, with $K$ in $Lat \A$, as a \emph{restriction algebra} for  $\A$, and refer to the weak$^*$-closure, $(\A|_K)^{-w^*}$, as a \emph{closed restriction algebra}.
In Theorem \ref{mainthrm} we showed that the strongly 
isometric representations of the partial Weyl commutation relations were in bijective correspondence with the restriction representations, $\rho_K$ say, for nonzero subspaces $K$ in $Lat \A_p$. We now show that nevertheless, the closed restriction algebras of $\A_p$, for $K$ a proper subspace, are all unitarily equivalent, and in fact are unitary equivalent to the Volterra nest algebra for the half-line.

We first note the following elementary unitary equivalences.

\begin{lemma}\label{l:intermediate} Let $K_1, K_2$ be proper closed subspaces which belong to one of the the following three subsets of $Lat \A_p$: (i)  $\N_v$, (ii)  $\N_a$, (iii) the subspaces $K_{\lambda, s}$, for $ \lambda\in \bR, s>0$.
Then the restriction algebras for $K_1$ and $K_2$ are unitarily equivalent.
\end{lemma}

\begin{proof}It is straightforward to check that
the restriction algebras of (i) (resp. (ii)) are unitarily equivalent to the restriction algebra $\A_p|_{L^2(\bR_+)}$ (resp. $\A_p|_{H^2(\bR)}$) by means of a unitary of the form $D_\mu$ (resp. $M_\lambda$). For the algebras in (iii) we first introduce the dilation unitaries $V_t, t\geq 0,$ defined by  $(V_tf)(x) =e^{t/2}f(e^t x)$. (This abuses earlier notation for WCR isometries but is consistent with the $M_\lambda, D_\mu, V_t$ notation of  \cite{kas-pow}, \cite{kat-pow-1}, \cite{kat-pow-2}.) We have the commutation relations
\[
V_tM_\lambda = M_{\lambda e^t}V_t, \quad V_tD_\mu = D_{\mu e^-t}V_t.
\]
In particular the unitary automorphism $Ad_{V_t}$ of $B(L^2(\bR))$ restricts to a unitary automorphism of $\A_p$.
Also, for $s_1, s_2>0$ and $t =\frac{1}{2}\log{{\frac{s_1}{s_2}}},$ we have $V_tM_{\phi_{s_2}}H^2(\bR) = M_{\phi_{s_1}}H^2(\bR)$.
It follows routinely from this that the algebras of (iii) are unitarily equivalent to the algebra 
$\A_p|_{\phi_1H^2(\bR)}$ by means of  unitaries of the form $M_\lambda D_\mu V_t$.
\end{proof}


The next lemma is well-known  and is a consequence of
the weak$^*$-density of the algebra of analytic trigonometric polynomials
in the $L^\infty(\bR_+)$. For completeness we give a proof which is also a prelude to the separation argument for the proof of Lemma \ref{l:weakstardense2}. 

\begin{lemma}\label{l:Mlambdadense}
Let $P_+$ be the orthogonal projection on $L^2(\bR_+)$.  Then the operator algebra $P_+{M}_{H^\infty}|_{L^2(\bR_+)}$ 
is $w^*$-dense in the maximal abelian  von Neumann algebra ${M}_{L^\infty(\bR_+)}$.
\end{lemma}
\begin{proof} 
 If $P_+{M}_{H^\infty} P_+$ is not dense in ${M}_{L^\infty(\bR_+)}$, then there exists an essentially bounded function $\phi$ supported on $\bR_+$ and  a $w^*$-continuous functional $\omega : B(L^2(\bR))\rightarrow\bC$, such that $\omega(P_+M_fP_+)=0$, for every bounded analytic function $f$, and $\omega(M_\phi)=1$. 
On the other hand, the restriction of $\omega$ to the multiplication algebra $M_{L^\infty(\bR)}$ induces a $w^*$-continuous functional on $L^\infty(\bR)$, which we also denote by $\omega$. Hence there exist $h\in L^1(\bR)$, such that
 \[\omega(f)=\int_\bR f(x) h(x)dx.\]
 Take $f(x)=e^{i\lambda x}$. Then 
\begin{align*}
\omega(P_+M_fP_+)=0\quad  &  \Leftrightarrow  
\quad \int_\bR e^{i\lambda x} \chi_{\mathbb{R_+}}(x) h(x)dx=0.
\end{align*}
Since this is true for all $\lambda\geq 0$ it follows that  $\chi_{R_+}  h$ lies in $H^1(\bR)$ and so is equal to
the zero function. Therefore the essential support of the function $h$ is contained in $\mathbb{R_-}$. Hence, given any function $\phi\in L^\infty(\bR_+)$
\begin{align*}
 \omega(M_\phi)=\int_\bR\phi(x)h(x)dx=0,
 \end{align*}
which is the desired contradiction.
\end{proof}

The Volterra nest algebra
$\A_{v+}$ on $L^2(\bR_+)$ is defined as the algebra of operators on this space which leaves invariant each of the subspaces $L^2[t, \infty)$, for $t\geq 0$.

\begin{lemma}\label{l:weakstardense1}
The restriction algebra  $\A_p|_{L^2(\bR_+)}$ is $w^*$-dense in  $\A_{v+}$.
\end{lemma}
\begin{proof}
By the previous lemma
the weak$^*$-closure of the restriction algebra 
contains the operators $\{M_\lambda, D_\mu :\lambda\in\bR, \mu\geq 0\}$. It is well-known that these unitaries generate $\A_{v+}$ as a the weak$^*$-closed operator algebra. Also, this can be deduced from the w*-density in $B(L^2(\bR))$
of the linear span of the products $M_\lambda D_\mu$, for $ \lambda, \mu \in \bR$.
\end{proof}

\begin{prop}\label{p:weakstardense2}
The restriction algebra  $\A_p|_{H^2(\bR_+)}$ is $w^*$-dense in
the nest algebra $Alg \N_{a+}$ on $H^2(\bR)$ for the nest 
\[
\N_{a+} = \{e^{i\lambda x}H^2(\bR): \lambda \geq 0\}\cup\{0\},
\]
and is unitarily equivalent to $\A_{v+}$.
\end{prop}

\begin{proof}
Let $F:H^2(\bR)\rightarrow L^2(\bR_+)$ be the restriction of the Fourier-Plancherel transform. Then,
\begin{align*}
F\A_p\big|_{H^2(\bR)} F^\ast P_+=\{ FAP_{H^2(\bR)}F^\ast P_+ : A\in\A_p\}= \{ FAF^\ast P_+ : A\in\A_p\}.
\end{align*}
Also, $FM_\lambda F^\ast= D_\lambda$ and $FD_\mu F^\ast= M_{-\mu}$. 
As in the proof of Lemma \ref{l:Mlambdadense} the algebra generated by the semigroup $\{M_{-\lambda}P_+:\lambda\geq 0\}$ is dense in  $M_{L^\infty(\bR_+)}$. Therefore it follows that
$F\A_p\big|_{H^2(\bR)} F^\ast|_{L^2(\bR_+)}$ is $w^*$-dense in $\A_{v+}$. 
Thus 
\begin{align*}
{(\A_p\big|_{H^2(\bR)})}^{-w^*}= 
F^\ast(\A_{v+})F\big|_{H^2(\bR)}
= Alg \mathcal{N}_a\big|_{H^2(\bR)},
\end{align*}
as required.
\end{proof}

We now fix $s>0$ and consider the case $K= K_s= \phi_s H^2(\bR)$.
Recall from the proof of Theorem \ref{t:conealgebras} that 
$M_{\phi_s}M_{s\mu}D_{\mu}M_{\phi_s}^\ast=e^{-is\mu^2/2}D_\mu$.
Applying $Ad(M_{\phi_s}^\ast)$ to $\A_p\big|_{K_s}$ we have 
the following identification of operator algebras on $H^2(\bR)$:
\begin{align*}
 M_{\phi_s}^\ast \A_p\big|_{K_s} M_{\phi_s}
 &= \{M_{\phi_s}^\ast A P_{K_s} M_{\phi_s}:A\in\A_p\}\\
 &= \{M_{\phi_s}^\ast AM_{\phi_s}P_{H^2(\bR)}:A\in\A_p\}\\&=
 \{AP_{H^2(\bR)}:A\in\A_{sp}\}
 \end{align*}
where $\A_{sp}$ is the weak$^*$-closed algebra generated by
$\{M_\lambda,M_{s\mu}D_\mu : \lambda,\mu\geq 0\}$.

Applying the Fourier transform we obtain that the algebra
\[
 \A_{sp}^\mathcal{F}\big|_{L^2(\bR_+)}:= F^*\A_{\rm sp}F\big|_{H^2(\bR)} 
\]
is generated as a weak$^*$-closed algebra by the set of isometries
\[
\{D_{\lambda}\big|_{L^2(\bR_+)}, D_{s\mu}M_{-\mu}\big|_{L^2(\bR_+)} : \lambda, \mu \geq 0\}.
\]

\begin{lemma}\label{l:weakstardense2}
 The algebra
$\A_{sp}^\mathcal{F}\big|_{L^2(\bR_+)}$ 
is dense  in
the Volterra nest algebra $\A_{v+}$.
\end{lemma}

\begin{proof}
Fix some $s,\mu>0$.
Let $\omega$ be a $w^*$-continuous functional 
\[\omega: B(L^2(\bR))\rightarrow \bR: T \mapsto \sum_k\langle T h_k,g_k\rangle\] 
for some $h_k,g_k\in L^2(\bR)$. Suppose that it annihilates 
$P_+\A_{sp}^\mathcal{F} P_+$, and therefore the operators 
$P_+D_{s\mu}M_{-\lambda}P_+$ for all $\lambda \in [0,\mu]$.
Then for these parameters we have
\begin{align*}
\omega(P_+ D_{s\mu}M_{-\lambda}P_+) = 0.
\end{align*}
Define the bounded linear functional
\[\omega_\mu: B(L^2(\bR))\rightarrow \bR: T \mapsto \omega(P_+D_{s\mu} T P_+).\]
Identify the restriction of $\omega_\mu$ on the multiplication algebra $M_{L^\infty(\bR)}$ with the $w^*$-continuous functional $\omega_\mu (f)=\int_{\bR} f(x) h_\mu(x) dx$, where $h_\mu$ is an $L^1(\bR)$ function. Then it follows from the definition of $\omega_\mu$ that $h_\mu$ is zero on $\bR_-$. 

Thus 
\[0=\omega(D_{s\mu}M_{-\lambda}P_+)=\omega_\mu (M_{-\lambda})= \int_{\bR} e^{-i\lambda x}h_\mu(x)dx. \]

Since $\int_{\bR} e^{-i\lambda x}h_\mu(x)dx=0$, it follows that $\hat{h}_\mu$ vanishes on the interval $(0,\mu)$. On the other hand,  $\hat{h}_\mu\in H^\infty(\bR)\cup {C}_0(\bR)$ and so $h_\mu = 0$. Thus for  $f \in L^\infty(\bR_+)$ we have  
\[\omega (M_f)=\lim\limits_{\mu\rightarrow 0} \omega (P_+D_{s\mu}M_fP_+)=\lim\limits_{\mu\rightarrow 0} \omega_\mu(M_f)=0.\]
It follows, by the usual separation principle, that the multiplication algebra for $L^\infty(\bR_+)$ must lie in the $w^*$-closure. Since the right shift operators also lie in the algebra it follows by standard arguments  that $\A_{v+}$ is contained in the closure, completing the proof.
\end{proof}

Combining the results of this section, we have the following,
\begin{thm}
Let $K$ be a proper invariant subspace of the parabolic algebra. Then 
${(\A_p|_K)}^{-w^*}$  is unitarily equivalent to the Volterra nest algebra $\A_{v+}$.
\end{thm}

\section{Quasicompact algebras}\label{s:quasicompact}

Let $\A$ be a weak$^\ast$-closed operator algebra on the Hilbert space $\H$ and $\K=\K(\H)$  the ideal of the compact operators. Define the \emph{quasicompact algebra} of $\A$ to be the C*-algebra $Q\A$ where
\begin{align*}
Q\A=\left(\A+ \K\right)
\cap\left(\A^\ast + \K\right).
\end{align*}
 Analogous algebras have been studied systematically in the theory of function spaces, the principal example being the nonseparable algebra of quasicontinuous functions,
\begin{align*} 
QC(\mathbb{T})=\left(H^\infty(\mathbb{T})+ C(\mathbb{T})\right)\cap \left(\overline{H^\infty(\mathbb{T})} + C(\mathbb{T})\right),
\end{align*}
where
$C(\mathbb{T})$ is the algebra of continuous functions on the unit circle (see \cite{dou}). 
Determining the structure of $Q\A$ and whether it differs from $\A\cap\A^\ast+ \K$ 
seems to be a rather deep problem in general. However for $\A=\A_v$ the following is well-known.

\begin{thm}\label{t:quasicompctvolterra}
The quasicompact algebra Q$\A_v$ is not equal to $\A_v\cap\A_v^\ast+\K$.
\end{thm}

 The proof of this theorem has two main ingredients. The first of these is that the triangular truncation operator with respect to the Volterra nest,
\begin{align*}
\mathcal{P}_v:\mathcal{C}_2(L^2(\bR))\rightarrow \A_v \cap\mathcal{C}_2(L^2(\bR)).
\end{align*}
is a contractive projection in the space $\mathcal{C}_2(L^2(\bR))$ of Hilbert-Schmidt operators which is  an unbounded operator with respect to the operator norm (see \cite{dav}). The second ingredient is to use this fact to create unit norm finite rank operators $A_k$ in $\A_v$, with orthogonal domains and ranges and with operator norms $\|A_k-A_k^*\|$ tending to zero. 
Then the infinite sum $A$ of the $A_k$ is a compact perturbation of $A^*$ which, furthermore, does not belong to $M_{L^\infty(\bR)}+\K$. 

In the proof of Theorem \ref{Jesussea} we adopt a similar strategy. However, in $\A_p$ there are no finite rank operators and we must make use of compact operators for the $A_k$. Also the orthogonality of domains and ranges must be replaced by an approximate form of this.



 
  

\begin{lemma}\label{triantrun}
The restriction of the triangular truncation operator $\mathcal{P}_v\big|_{\A_{p}+\A_{p}^\ast}$ is unbounded.
\end{lemma}
\begin{proof}
Let $p_n$ be a real coefficient polynomial on $\mathbb{T}$ with supremum norm 1, such that the polynomials $f_n(z)=p_n(z)-\overline{p_n(z)}$
 satisfy the property $\|f_n\|_\infty \rightarrow 0$. For example, take
 \begin{align*}
 p_n(z)=c_n\sum_{k=1}^{n}\frac{1}{k}z^k.
 \end{align*}
 for appropriate constants $c_n$.
Let $Z$ be a unitary operator in $\A_{p}$ with full spectrum, such as $M_1$. Take $(F_n)$ to be a bounded approximate identity of Hilbert-Schmidt operators in the unit ball of $\A_{p}$. Then $F_n^\ast$ is also a bounded approximate identity on the space of Hilbert-Schmidt operators in $\A_{p}^\ast$. 

By the functional calculus $\|p_n(Z)\|=\|p_n\|= 1$ and there  exists a sequence $(\xi_n)$ in the unit sphere of $L^2(\bR)$, such that $\|p_n(Z)\xi_n\|>2/3$, for every $n\in\bN$. Fix $n\in\bN$.  Then   $p_n(ZF_m)\xi_n\rightarrow p_n(Z)\xi_n$ as $m\to \infty$. Choose inductively a subsequence $(F_{m_n})$,  denoted  $(F_n)$, such that
\begin{align*}
\|p_n(ZF_n)\xi_n\|>1/2.
\end{align*}
Since $p_n(ZF_n)$ belongs to $\A_{p}$,  we have $\langle K, p_n(ZF_n)^\ast\rangle_{H-S}=0$, for every Hilbert-Schmidt operator $K\in\A_v$ and $n\in\bN$. Thus
\begin{align*}
\|\mathcal{P}_v(p_n(ZF_n)-p_n(ZF_n)^\ast)\|=&
\|\mathcal{P}_v(p_n(ZF_n))-\mathcal{P}_v(p_n(ZF_n)^\ast)\|=\\&=
\|p_n(ZF_n)\|\geq\|p_n(ZF_n)\xi_n\|>1/2.
\end{align*}

On the other hand, since $ZF_n$ is a contraction for every $n\in\bN$, von Neumann's inequality
yields
\begin{align*}
\|p(ZF_n)+q(ZF_n)^\ast\|\leq \|p+\overline{q}\|
\end{align*}
for all polynomials $p,q$ in the disc algebra. Taking $p=p_n$ and $q=-p_n$, it follows 
\begin{align*}
\|p_n(ZF_n)-p_n(ZF_n)^\ast\|\leq \|p_n-\overline{p_n}\|\rightarrow 0,
\end{align*}
which completes the proof.
\end{proof}

\begin{thm}\label{Jesussea}
The quasicompact algebra $Q\A_{p}$ is strictly larger than the algebra\\ $\A_{p}\cap \A_{p}^\ast +\K=\bC I+\K$.
\end{thm}
\begin{proof}
Recall the operators $p_n(ZF_n)$ from the proof of Lemma \ref{triantrun}, where $Z\in \A_p$ is unitary with spectrum the unit circle, $(F_n)$ is the bounded approximate identity in $\A_p$ of Hilbert-Schmidt operators, and $p_n(z)$ are polynomials in $C(\bT)$, of unit $\|\cdot \|_\infty$-norm, with real parts converging uniformly to
zero. Also we have $1/2\leq\|p_n(ZF_n)\|\leq 1$.

 Since these operators are compact, there exist compact intervals $K_n$ such that for any  $f\in L^2(\bR)$,
\begin{equation}\label{compest}
\|p_n(ZF_n)f\|\leq \|P_{K_n} f\|+ \frac{1}{2^n}\|f\|
\end{equation}
and
\begin{equation}\label{compest2}
\|P_{\bR\backslash K_n}p_n(ZF_n)f\|\leq  \frac{1}{2^n}\|f\|,
\end{equation}
where $P_{K_n}$ is the projection on $L^2(K_n)$. We can also arrange that for all $n$,
\begin{equation}\label{compest3}
\|\left(p_n(ZF_n)-p_n(ZF_n)^\ast\right)f\|\leq \|p_n(ZF_n)-p_n(ZF_n)^\ast\|\,\left(\|P_{K_n} f\|+ \frac{1}{2^n}\|f\|\right)
\end{equation}
and
\begin{equation}\label{compest4}
\|P_{\bR\backslash K_n}\left(p_n(ZF_n)-p_n(ZF_n)^\ast\right)f\|\leq  \frac{1}{2^n}
 \|p_n(ZF_n)-p_n(ZF_n)^\ast\|\,\|f\|.
\end{equation}
  
We now choose $(t_n)$ to be an increasing sequence so that the translates  $\Lambda_n=K_n+t_n$
are disjoint sets, with
$\max\Lambda_n<\min\Lambda_{n+1}$ for all $n$.
We also write $\Lambda_0=\bR\backslash \cup_{n=1}^\infty \Lambda_n$. Since the projection of triangular truncation with respect to the binest commutes with $Ad_{D_t}$, it follows that the operators
\begin{align*}
A_n=D_{t_n}(p_n(ZF_n))D_{t_n}^\ast
\end{align*}
lie in $\A_{p}$. 

\noindent \underline{Claim 1:}\, Given $f\in L^2(\bR)$, the sequence $(\sum\limits_{k=1}^{n}A_k f)_{n}$ is convergent.
 
To prove this claim, it suffices to show that the given sequence is a Cauchy sequence. 
Let $C$ be the compact set $\cup_{m=n}^N \Lambda_m$.
Then 
\begin{align*}
\bigg\|\sum_{k=n}^N A_k f\bigg\|^2=
\int_{\bR\backslash C} \bigg|\sum_{k=n}^N A_k f\bigg|^2 +
\int_C \bigg|\sum_{k=n}^N A_k f\bigg|^2.
\end{align*}
Estimating the integrals we have 
\begin{align*}
 \int_{\bR\backslash C} \bigg|\sum_{k=n}^N A_k f\bigg|^2&=
\bigg\| P_{\bR\backslash C} \sum_{k=n}^N A_k f\bigg\|^2 \leq
\left(\sum_{k=n}^N \| P_{\bR\backslash C}A_k f\|\right)^2\leq
\left(\sum_{k=n}^N \frac{1}{2^k} \|f\|\right)^2 =
\left(\sum_{k=n}^N \frac{1}{2^k}\right)^2 \|f\|^2.
\end{align*}
Also,
\begin{align*}
\,\,\,\,\,\,\int_{C} \bigg|\sum_{k=n}^N A_k f\bigg|^2&=
\int_{\bR} \bigg|\sum_{m=n}^N P_{\Lambda_m}\sum_{k=n}^N A_k f\bigg|^2\leq
 2 \int_{\bR} \bigg|\sum_{m=n}^N P_{\Lambda_m} A_m f\bigg|^2 +
2 \int_{\bR} \bigg|\sum_{m=n}^N \sum_{\substack{k=n\\k\neq m}}^N  P_{\Lambda_m} A_k f\bigg|^2.
\end{align*}
The first  term here gives 

\begin{align*}
\int_{\bR} \bigg|\sum_{m=n}^N P_{\Lambda_m} A_m f\bigg|^2=
\sum_{m=n}^N \|P_{\Lambda_m} A_m f\|^2 &\leq \sum_{m=n}^N ( \|P_{\Lambda_m} f\| +\frac{1}{2^m}\|f\|)^2
\\&\leq  \sum_{m=n}^N \left( \|P_{\Lambda_m} f\|^2 + \left(\frac{2}{2^m} + \frac{1}{2^{2m}}\right)\|f\|^2\right)\\
&\leq  \|P_{C} f\|^2 + \left(\sum_{m=n}^N \left(\frac{2}{2^m} + \frac{1}{2^{2m}}\right) \right) \|f\|^2.
\end{align*}
Note that for every $\epsilon_1>0$,  we can choose $n_0$ big enough so that 
$\|P_{A} f\|\leq \epsilon_1\|f\|$, where $A=\cup_{m=n_0}^\infty \Lambda_m$.
 
For the second term, it follows by relation \eqref{compest2} that
\begin{align*}
\int_{\bR} \bigg|\sum_{m=n}^N \sum_{\substack{k=n\\k\neq m}}^N  P_{\Lambda_m} A_k f\bigg|^2&=
\bigg\| \sum_{m=n}^N \sum_{\substack{k=n\\k\neq m}}^N  P_{\Lambda_m} A_k f\bigg\|^2=
\bigg\| \sum_{k=n}^N \sum_{\substack{m=n\\m\neq k}}^N  P_{\Lambda_m} A_k f\bigg\|^2\\&=
\bigg\| \sum_{k=n}^N   P_{C\backslash\Lambda_k} A_k f\bigg\|^2\leq 
\left( \sum_{k=n}^N  \| P_{C\backslash\Lambda_k} A_k f\|\right)^2 \\
&\leq 
\left( \sum_{k=n}^N \frac{1}{2^k}\|f\|\right)^2 = \left( \sum_{k=n}^N \frac{1}{2^k}\right)^2
\|f\|^2.
\end{align*}

Combining the above estimates we get
\begin{align*}
\bigg\|\sum\limits_{k=n}^N A_k f\bigg\|^2\leq \left(3 \left(\sum_{k=n}^N \frac{1}{2^k}\right)^2 
+2 \sum_{m=n}^N \left(\frac{2}{2^m} + \frac{1}{2^{2m}}\right)+ 2\epsilon_1^2\right) \|f\|^2.
\end{align*}
Hence, there exists $n_0\in \bN$ such that $\|\sum\limits_{k=n}^N A_k f\|^2\leq \epsilon$, 
for all $n,N>n_0$, proving the first claim. 

\noindent \underline{Claim 2:} The sequence  $\{\sum\limits_{k=1}^{n}A_k\}_{n}$ is uniformly bounded.

Let $f\in L^2(\bR)$. Then 
\begin{align*}
\bigg\|\sum_{k=1}^n A_k f\bigg\|^2&= \lim\limits_{M\rightarrow \infty}
\bigg\|\sum_{m=0}^M P_{\Lambda_m}\sum_{k=1}^n A_k f\bigg\|^2.
\end{align*}
Also
\begin{align*}
\bigg\|\sum_{m=0}^M P_{\Lambda_m}\sum_{k=1}^n A_k f\bigg\|^2\leq
2\bigg\|\sum_{k=1}^n P_{\Lambda_k}A_k f\bigg\|^2+2 \bigg\|
\sum_{k=1}^n\sum_{\substack{m=0 \\ m \neq k}}^M P_{\Lambda_m} A_k f\bigg\|^2. 
\end{align*}
Applying now the relation \eqref{compest} we obtain
\begin{align*}
 \bigg\|\sum_{k=1}^n P_{\Lambda_k}A_k f\bigg\|^2&\leq 
\sum_{k=1}^n \left(\|P_{\Lambda_k}f\|+ \frac{1}{2^k}\|f\|\right)^2\\&=
\sum_{k=1}^n \|P_{\Lambda_k}f\|^2+ \sum_{k=1}^n \left(\frac{2}{2^k}+\frac{1}{2^{2k}}\right)\|f\|^2\leq 4\|f\|^2.
\end{align*}
Moreover
\begin{align*}
\bigg\|
\sum_{k=1}^n\sum_{\substack{m=0 \\ m \neq k}}^M P_{\Lambda_m} A_k f\bigg\|^2\leq
\left(\sum_{k=1}^n\bigg\|\sum_{\substack{m=0 \\ m \neq k}}^M P_{\Lambda_m} A_kf\bigg\|\right)^2\leq&
\left(\sum_{k=1}^{n}\| P_{\bR\backslash \Lambda_k} A_k f\|\right)^2\leq
\left(\sum_{k=1}^{n}\frac{1}{2^k}\| f\|\right)^2\\&= \left(\sum_{k=1}^{n}\frac{1}{2^k}\right)^2\| f\|^2\leq
\|f\|^2.
\end{align*}
Hence the norms $\|\sum\limits_{k=1}^n A_k\|$ are indeed uniformly bounded, as required. 

Let $M=\sup_n \|\sum\limits_{k=1}^n A_k\|<\infty$ and define the linear operator $A$ acting on $L^2(\bR)$ by the formula 
\begin{align*}
Af= \lim_n \sum_{k=1}^{n} A_k f.
\end{align*}
By the first claim  $A$ is well-defined and  
by the second claim, $\|Af\|\leq M\|f\|$, for $f\in L^2(\bR)$. Thus  $A$ is a bounded operator equal to the SOT-limit of the sequence
$\{\sum\limits_{k=1}^{n}A_k\}_{n}$, and so lies in $\A_p$. Note also that $A$ is not a compact operator since $\|A_k\|\geq 1/2$ for all $k$.

Define now the Hilbert-Schmidt operators 
\begin{align*}
X_n=A_n-A_n^\ast=D_{t_n}(p_n(ZF_n)-p_n(ZF_n)^\ast)D_{t_n}^\ast
\end{align*}
and note that $\|X_n\|\rightarrow 0$.
By the calculations above we see that the sequence of the partial sums of $\sum\limits_{n=1}^{\infty} X_n$ is Cauchy with respect to the operator norm and so the norm limit $X=\sum\limits_{n=1}^{\infty} X_n$ 
is a compact operator  on $L^2(\bR)$.
Since involution is continuous in the weak operator topology, we see that $A-A^\ast=X$. Thus
$A$ lies in the quasicompact algebra $Q\A_p$ and it remains to show that $A\notin \bC I+\K$.

Assume that this is not true. 
Since the algebra $\bC I+\K$ is norm closed we have $A=cI+K$ with $c\in\bC$ and $K\in \K$. Left multiply both sides by the projection $P_{\Lambda_0}$. Since multiplication is separately SOT-continuous we see that 
$P_{\Lambda_0} A$ is the SOT-limit of the operators $\{\sum\limits_{k=1}^{n}P_{\Lambda_0}A_k\}_{n}$. Moreover we have the estimates 
\begin{align*}
\|P_{\Lambda_0}A_k\|\leq \frac{1}{2^k}, \text{ for every } k\in\bN,
\end{align*}
so the above sequence converges uniformly to $P_{\Lambda_0} A$. Therefore $P_{\Lambda_0} A$ is a compact operator, and so $cP_{\Lambda_0} $ is also compact. 
However, ${\Lambda_0}$ is a union of nonempty intervals and so 
$c=0$. But this implies that $A$ is compact, which gives the desired contradiction.
\end{proof}

\begin{rem}
 We remark that Q$\A_p$, like $QC$ and also Q$\A_v$, is a nonseparable C*-algebra. Indeed one can determine a well-separated uncountable set by considering  operators of the form  $\sum_{k=1}^\infty b_kA_k$ where $(b_k)$ is a 0-1 sequence and $(A_k)$ is a sufficiently  approximately orthogonal sequence of unit norm operators in $\A_p$ as in the proof above.
\end{rem}
\begin{rem}
As we noted in Section \ref{s:parabolicalgebra}, the parabolic algebra $\A_p$ is equal to the intersection  $\A_v \cap \A_a$. It  seems plausible, given the proof above, that the \emph{essential parabolic algebra}  $\A_p/\K$ is a proper subalgebra of the intersection algebra
$(\A_v/\K)\cap (\A_a/\K)$ however we do not know if this is the case. This is equivalent to determining whether $(\A_v+\K)\cap (\A_a +\K)$ is strictly larger than $\A_p +\K$.
\end{rem}

\subsection{Further directions}\label{ss:openprobs}
The parabolic algebra sits between the very well-understood algebra
$H^\infty(\bR)$ and the equally well-understood Volterra nest algebra $\A_v$. Determining further algebraic and geometric properties of $\A_p$ is likely to require methods from both commutative and noncommutative perspectives. In particular, the following 3 problems seem to be rather deep requiring both perspectives.

\medskip

{ Question 1.} \emph{How are the weakly closed ideals of $\A_p$ characterised ?} The weakly closed ideals of a nest algebra $\A$ are determined by left-continuous order homomorphisms from $Lat \A$ to $Lat \A$ (\cite{erd-pow}, \cite {dav}). Similar such boundary functions, from $Lat \A_p$ to $Lat \A_p$,  are likely to play a role in resolving this question.
\medskip

{ Question 2.} \emph{What is the Jacobson radical of $\A_p$ ?} In particular is there some kind of analogue of Ringrose's characterisation (\cite{rin}, \cite{dav}) ?
\medskip

{ Question 3.} \emph{Is there a variant of Arveson's distance formula for $\A_p$ ?} Arveson's distance formula (\cite{dav}, \cite{pow-distance})
for a nest algebra is closely related to distance formulae for $H^\infty(\bR)$ so in some ways this is a natural problem for $\A_p$. Also it would lead to the hyperreflexivity of $\A_p$, a property known for both $\A_v$ and for
$H^\infty(\bR)$ as an operator algebra on $H^2(\bR)$ (Davidson \cite{dav-distance}).

\medskip

{ Question 4.} \emph{Does $\A_p$ have zero divisors ?} We suspect that this intriguing question is less deep and indeed we believe that there are zero divisors even in the norm-closed algebra, $\B_p$ say, that is generated by the 2 semigroups. This algebra is analysable as a norm-closed semi-crossed product \cite{kas-normclosed} and elements admit natural generalised Fourier series. First let us remark that the unclosed algebra has {no} divisors of zero simply because operators in the algebra have generalised {finite} Fourier series of the form
\[
\sum_{\lambda_{k}\in \bR} A_kM_{\lambda_k}
\]
where each $A_k$ is a finite linear combination of the $D_\mu$ for $\mu\in \bR$.
Thus a product of nonzero elements has a nonzero first Fourier coefficient and so is also nonzero. On the other hand is it possible to construct, in some manner,  2 absolutely norm-convergent generalised Fourier series (with no first nonzero coefficient) such that the product is zero ?

\end{document}